\documentclass[11pt]{article}

\usepackage[margin=1in]{geometry}
\usepackage[hidelinks]{hyperref}
\hypersetup{
  pdftitle={Galois-Orbit Structure, Ramanujan Sums, and Stable-Range Collapse for Cyclotomic Cosine Formulas},
  pdfauthor={Juan D. Velez and Carlos A. Cadavid},
  pdfkeywords={cyclotomic cosine sums, Galois orbits, stable range, Ramanujan sums, compatible polynomial families}
}
\usepackage{microtype}
\usepackage{xcolor}
\usepackage{comment}

\usepackage{amsmath,amssymb,amsthm,mathtools}
\usepackage{enumitem}

\usepackage[nameinlink,noabbrev]{cleveref}

\theoremstyle{plain}
\newtheorem{theorem}{Theorem}[section]
\newtheorem{proposition}[theorem]{Proposition}
\newtheorem{corollary}[theorem]{Corollary}
\newtheorem{lemma}[theorem]{Lemma}

\theoremstyle{definition}
\newtheorem{definition}[theorem]{Definition}

\theoremstyle{remark}
\newtheorem{remark}[theorem]{Remark}

\theoremstyle{definition}
\newtheorem{example}[theorem]{Example}

\newcommand{\Q}{\mathbb{Q}}
\newcommand{\Z}{\mathbb{Z}}

\newcommand{\Gal}{\mathrm{Gal}}

\title{Galois-Orbit Structure, Ramanujan Sums, and Stable-Range Collapse for Cyclotomic Cosine Formulas}
\author{%
Juan D.\ V\'elez\thanks{Corresponding author.}\\
{\small Departamento de Matem\'aticas, Sede Medell\'in, Universidad Nacional de Colombia, Medell\'in, Colombia}\\
{\small \href{mailto:jdvelez@unal.edu.co}{\texttt{jdvelez@unal.edu.co}}}
\and
Carlos A.\ Cadavid\\
{\small Departamento de Matem\'aticas, Universidad EAFIT, Medell\'in, Colombia}\\
{\small \href{mailto:ccadavid@eafit.edu.co}{\texttt{ccadavid@eafit.edu.co}}}%
}
\date{September 2026}

\begin{document}

\maketitle

\begin{abstract}
We study cyclotomic cosine evaluations beyond the fully symmetric setting by combining Galois-orbit structure with finite-frequency expansions. First, for truncation-compatible polynomial families over \(\Q\), we distinguish invariance of the evaluated value from invariance of the polynomial formula. Summability by a rational function of the level forces eventual rationality of the evaluated values. Conversely, every rational-valued sequence can be realized by a compatible family whose total degree is uniformly bounded by two, showing that value-level Galois invariance alone carries essentially no stable-range rigidity. We then introduce explicit orbit-structured pattern sums indexed by multiplicative configurations. For every fixed pattern, at each sufficiently large admissible level \(n\), its cyclotomic cosine evaluation collapses to an affine function
\[
n\,\mathsf S_0(\mathbf r,\mathbf h)-2^{|\mathbf h|},
\]
where \(\mathsf S_0\) is the zero-frequency coefficient of the associated Laurent polynomial. Consequently, every family uniformly generated from finitely many such patterns by a single polynomial recipe has eventually polynomial evaluation. Finally, decomposing the punctured index set into gcd-orbits yields an exact orbitwise formula in terms of Ramanujan sums. The semiprime and prime-power cases become transparent specializations of this general identity. These results identify both the obstruction to extending symmetric rigidity from value-level Galois invariance and a concrete non-symmetric class exhibiting uniform stable-range collapse.


\end{abstract}

\section{Introduction}
\label{sec:introduction}

In \cite{VelezCadavidPartI} we introduced an inverse-limit formalism for truncation-compatible polynomial families evaluated at the punctured cyclotomic cosine points
\[
\alpha_{k,n}=\cos\!\Bigl(\frac{2\pi k}{n}\Bigr),
\qquad 1\le k\le n-1,
\]
with the level specialization \(z=n-1\). The guiding principle of that paper was that an \(n\)-dependent cyclotomic construction should not be treated as a collection of unrelated fixed-level expressions, but rather as a single compatible algebraic object. Within that framework, bounded-degree symmetry forced a strong stabilization phenomenon: if \(F=(F_m)_{m\ge1}\) is truncation-compatible, symmetric, and of uniformly bounded total \(x\)-degree, then for all sufficiently large \(n\) its cosine-point evaluation factors through finitely many universal punctured cosine power sums
\[
P_h(n)=\sum_{k=1}^{n-1}(2\alpha_{k,n})^h.
\]
In particular, one obtains eventual polynomiality of the corresponding evaluations. We refer to \cite{VelezCadavidPartI} for the precise inverse-limit framework, the stable-range rigidity theorem, and the extensions there to product-structured families. For general background on cyclotomic fields, see \cite{Washington1997}; related evaluations of trigonometric character sums by discrete heat-kernel methods appear in \cite{CHJSV23}.

The present paper begins from that structural picture and asks how much of it survives once one moves beyond the symmetric regime. In Part~I, the decisive hypothesis was full permutation symmetry in the variables
\[
(x_1,\dots,x_{n-1}),
\]
which made the stable-range reduction to finitely many additive invariants possible. But the punctured cyclotomic cosine configuration also carries a second symmetry, arithmetic rather than combinatorial in nature: the action of the cyclotomic Galois group
\[
G_n=\Gal(\Q(\omega_n)/\Q)\cong (\Z/n\Z)^\ast,
\qquad \omega_n=e^{2\pi i/n},
\]
which acts on the cosine points by
\[
\sigma_r(\alpha_{k,n})=\alpha_{rk,n},
\]
with the index understood modulo \(n\). The central question of this sequel is therefore the following:

\medskip

\noindent\emph{What remains of the summability theory of \cite{VelezCadavidPartI} once full symmetric-group symmetry is abandoned and one retains only orbit structure arising from the cyclotomic Galois action?}

\medskip

Beyond full symmetry, one must distinguish two notions that are largely invisible in the symmetric theory because there the formula-level structure is so strong. For a truncation-compatible family \(F=(F_n)_{n\ge1}\), one may ask whether the evaluated value
\[
a_F(n)=F_{n-1}(\alpha_{1,n},\dots,\alpha_{n-1,n})
\]
is fixed by \(G_n\), or instead whether the polynomial
\[
F_{n-1}(x_1,\dots,x_{n-1})
\]
itself is invariant under the induced permutation action of \(G_n\) on the punctured index set \(\{1,\dots,n-1\}\). These are very different conditions. The first is an arithmetic statement about the field of definition of the evaluated number; the second is a structural statement about the formula before evaluation. One of the main points of the present paper is that, once full symmetry is removed, this distinction becomes decisive.

The paper has two complementary parts.

The first part is negative and remains in the global compatible-family framework developed in Part~I. We prove that summability by a rational function of the level still has an unavoidable arithmetic shadow: if a compatible family is eventually equal, at cyclotomic cosine points, to a rational function \(R(n)\in\Q(n)\), then for all sufficiently large \(n\) its evaluated value is fixed by \(G_n\). Equivalently, eventual summability forces eventual rationality of the evaluated values. However, this necessary condition is far from sufficient. Indeed, we prove that every rational-valued sequence \((b_n)_{n\ge2}\) can be realized as
\[
b_n=F_{n-1}(\alpha_{1,n},\dots,\alpha_{n-1,n})
\]
for some truncation-compatible family \(F\), and in fact this can be done with \(\deg F_m\le 2\) for every \(m\). Thus value-level \(G_n\)-invariance, although necessary for summability, is far too weak to imply any analogue of the stable-range rigidity theorem of \cite{VelezCadavidPartI}, even under a uniform quadratic degree bound.

The second part of the paper is positive but is formulated in a different framework. Instead of working with one compatible family across all integer levels, we pass to explicit orbit-structured pattern sums indexed by multiplicative configurations. The negative results show that no meaningful rigidity theory can be recovered from value-level Galois invariance alone; one must instead impose structural orbit conditions directly on the formulas.

For a fixed pattern \((\mathbf r,\mathbf h)\), the total cyclotomic cosine evaluation can be analyzed directly, without first decomposing into \(G_n\)-orbits. After binomial expansion, the evaluation reduces to finitely many geometric sums indexed by integer frequencies. In the stable range, namely when the level \(n\) is larger than the maximal possible absolute frequency and when all multipliers \(r_j\) are invertible modulo \(n\), every nonzero frequency is too small to be divisible by \(n\). As a result, every nonzero mode contributes the same constant \(-1\), while the zero mode contributes \(n-1\). This yields the all-level collapse formula
\[
T^{(n)}_{\mathbf r,\mathbf h}(\alpha_{1,n},\dots,\alpha_{n-1,n})
=
n\,\mathsf S_0(\mathbf r,\mathbf h)-2^{|\mathbf h|}
\]
for every sufficiently large admissible level \(n\). Thus the prime case is not the endpoint of the theory, but only the first visible instance of a broader stable-range phenomenon.

From this general collapse theorem one immediately obtains eventual polynomiality for every family uniformly generated from finitely many such pattern sums through a single polynomial recipe. The resulting positive theory is therefore stronger than a prime-tail statement: it applies to all sufficiently large admissible levels \(n\), not merely to primes or to special factorization types.

The orbitwise analysis leads naturally to the classical Ramanujan sums
\[
c_q(M)=\sum_{u\in(\Z/q\Z)^\ast}e^{2\pi iMu/q},
\]
introduced by Ramanujan in his study of trigonometrical sums \cite{Ramanujan1918} and developed further in the classical literature; see, for example, \cite{Hardy1921,Apostol1976,McCarthy1986,SchwarzSpilker1994}. Their close connection with cyclotomic polynomials has also been studied explicitly; see \cite{Toth2010}. In the present setting these sums arise canonically from restricting the finite-frequency expansion to individual gcd-orbits.

Even so, the orbitwise decompositions remain important. They are no longer needed to prove the total collapse formula, but they reveal how the total contribution is internally distributed according to the arithmetic shape of the level. For a general divisor orbit, the contribution is expressed exactly through Ramanujan sums; the semiprime and prime-power cases are especially transparent consequences. At semiprime levels \(n=pq\), the total collapse splits into three \(G_{pq}\)-orbit contributions,
\[
\Omega_{pq}
=
\mathcal O_1(pq)\sqcup \mathcal O_p(pq)\sqcup \mathcal O_q(pq),
\]
with the two non-unit orbits reducing to prime-level cases and the unit orbit governed by a Ramanujan-sum contribution. At prime-power levels \(n=p^N\), the punctured index set decomposes into a tower of gcd-orbits
\[
\Omega_{p^N}
=
\bigsqcup_{s=0}^{N-1}\mathcal O_{p^s}(p^N),
\]
whose deepest boundary orbit reduces to the prime case, while the inner orbits contribute only zero-frequency mass. Thus the prime, semiprime, and prime-power analyses survive in the final paper not as isolated main theorems, but as orbitwise refinements of the stronger all-level collapse theorem.

Accordingly, the present paper should be read neither as a direct replacement for the symmetric theory of \cite{VelezCadavidPartI} nor as a full all-level orbit-theoretic extension of it. Rather, it does two more precise things. First, it identifies a sharp obstruction to naive Galois-theoretic extensions of the symmetric theory: value-level Galois invariance is necessary but carries essentially no rigidity. Second, it exhibits an explicit non-symmetric orbit-structured model class for which a broad all-level stable-range collapse theorem holds, together with refined orbitwise descriptions in the semiprime and prime-power cases.

The logical structure of the paper is as follows. In Section~\ref{sec:summability-Gn}, we remain in the compatible-family framework and prove that summability forces eventual value-level \(G_n\)-invariance, together with the sharp non-converse. In Section~\ref{sec:Gp-patterns}, we introduce the explicit orbit-structured pattern sums and the corresponding finite-generation framework used in the positive theory. In Section~\ref{sec:pattern-collapse}, we prove the main positive theorem: the all-level stable-range collapse formula for every sufficiently large admissible level \(n\). We then deduce eventual polynomiality for uniformly pattern-generated families, derive a general Ramanujan-sum formula for individual gcd-orbits, and record the prime, semiprime, and prime-power cases as transparent refinements. Finally, in Section~\ref{sec:conclusion}, we discuss the remaining composite-level orbit-elimination problem and indicate several directions for further work, including twisted orbit sums and rational-function extensions.

\medskip

\noindent\textbf{Scope of the positive theorems.}
The positive results of the present paper are not full all-level extensions of the compatible-family rigidity theorem of \cite{VelezCadavidPartI}. They concern an explicit class of orbit-structured pattern sums and the families generated from them. Within that class, however, the main collapse theorem is genuinely all-level in the stable range: it applies to every sufficiently large admissible level \(n\), not merely to primes or semiprimes.

\medskip

A secondary aim of the paper is to make visible the arithmetic obstruction that still remains at the orbitwise level. In the symmetric theory of \cite{VelezCadavidPartI}, the finitely many invariants \(P_h(n)\) controlled the stable range because full symmetry erased the internal orbit structure of the punctured index set. In the present paper, the total stable-range formula can still be recovered without tracking separate orbits. For the explicit pattern class studied here, the Ramanujan-sum formula obtained below controls each gcd-orbit at every admissible level. The broader problem persists in a finer form: under what structural hypotheses, beyond this model class, can one control not only the total collapse, but also the separate orbit contributions uniformly across levels?

\section{Summability and the limits of value-level \texorpdfstring{$G_n$}{Gn}-invariance}
\label{sec:summability-Gn}

We begin in the same global framework as in \cite{VelezCadavidPartI}, namely that of truncation-compatible polynomial families. The purpose of the present section is negative and foundational. We isolate the weakest unavoidable Galois-theoretic consequence of summability at cyclotomic cosine points, and then show that this consequence is far too weak to imply any rigidity by itself.

The main point is that, once full symmetric-group symmetry is abandoned, one must distinguish carefully between two different notions of \(G_n\)-invariance: invariance of the \emph{evaluated value} and invariance of the \emph{polynomial formula}. Summability forces the former, but the former alone imposes essentially no asymptotic structure.

\subsection{Compatible families and cosine-point evaluation}

For simplicity, the present sequel is formulated over \(\Q\). The main new issue here is not coefficient-field generality, but rather the contrast between arithmetic invariance and structural orbit symmetry.

For each \(n\ge 2\), let
\[
\omega_n=e^{2\pi i/n},
\qquad
G_n=\Gal(\Q(\omega_n)/\Q)\cong (\Z/n\Z)^\ast,
\]
where \(\sigma_r\in G_n\) is determined by
\[
\sigma_r(\omega_n)=\omega_n^{\,r},
\qquad (r,n)=1.
\]
Write
\[
\alpha_{k,n}=\cos\!\Bigl(\frac{2\pi k}{n}\Bigr)
=\frac{\omega_n^k+\omega_n^{-k}}{2},
\qquad 1\le k\le n-1.
\]
The Galois action permutes these cosine values according to
\[
\sigma_r(\alpha_{k,n})=\alpha_{\sigma_r(k),\,n},
\]
where \(\sigma_r(k)\) denotes the unique element of \(\Omega_n=\{1,\dots,n-1\}\) congruent to \(rk\pmod n\).

Throughout this section, a \emph{compatible family} means a sequence
\[
F=(F_n)_{n\ge 1},
\qquad
F_n\in \Q[x_1,\dots,x_n],
\]
satisfying
\[
F_n(x_1,\dots,x_n)=F_{n+1}(x_1,\dots,x_n,0)
\qquad (n\ge 1).
\]
For \(n\ge 2\), its cosine-point evaluation is
\[
a_F(n)=F_{n-1}(\alpha_{1,n},\dots,\alpha_{n-1,n}).
\]

\subsection{Two notions of \texorpdfstring{$G_n$}{Gn}-invariance}

The first distinction is conceptual and will remain important throughout the paper.

\begin{definition}[Value-level \texorpdfstring{$G_n$}{Gn}-invariance]
\label{def:value-level-invariance}
Let \(F=(F_n)_{n\ge 1}\) be a compatible family.
We say that \(F\) is \emph{value-level \(G_n\)-invariant at level \(n\)} if
\[
a_F(n)=F_{n-1}(\alpha_{1,n},\dots,\alpha_{n-1,n})
\]
is fixed by \(G_n\). Equivalently,
\[
a_F(n)\in \Q.
\]
\end{definition}

\begin{definition}[Polynomial \texorpdfstring{$G_n$}{Gn}-symmetry]
\label{def:polynomial-Gn-symmetry}
Let \(F=(F_n)_{n\ge 1}\) be a compatible family.
We say that \(F\) is \emph{polynomially \(G_n\)-symmetric at level \(n\)} if
\[
F_{n-1}(x_1,\dots,x_{n-1})
=
F_{n-1}(x_{\sigma_r(1)},\dots,x_{\sigma_r(n-1)})
\qquad\text{for all }\sigma_r\in G_n.
\]
\end{definition}

\begin{remark}
\label{rem:two-notions}
These notions should not be confused. Value-level \(G_n\)-invariance is a condition on the evaluated number
\[
F_{n-1}(\alpha_{1,n},\dots,\alpha_{n-1,n}),
\]
hence on its arithmetic field of definition. Polynomial \(G_n\)-symmetry, by contrast, is a condition on the formula before evaluation. In the symmetric theory of \cite{VelezCadavidPartI}, the distinction is largely hidden by the strength of the structural hypotheses. Beyond full symmetry, however, it becomes decisive.
\end{remark}

\subsection{Summability forces eventual rationality}

We now isolate the weakest unavoidable arithmetic consequence of summability.

\begin{definition}[Summability at cosine points]
\label{def:summability-rational}
A compatible family \(F=(F_n)_{n\ge 1}\) is said to be \emph{summable at cosine points} if there exist an integer \(N\ge 2\) and a rational function
\[
R(n)\in \Q(n),
\]
defined at every integer \(n\ge N\), such that
\[
a_F(n)=R(n)
\qquad\text{for all integers }n\ge N.
\]
\end{definition}

\begin{lemma}[Summability forces eventual value-level \texorpdfstring{$G_n$}{Gn}-invariance]
\label{lem:summability-implies-value-invariance}
Let \(F=(F_n)_{n\ge 1}\) be a compatible family.
Assume that \(F\) is summable at cosine points. Then for every sufficiently large \(n\), the value \(a_F(n)\) is fixed by \(G_n\). Equivalently,
\[
a_F(n)\in \Q
\]
for all sufficiently large \(n\).
\end{lemma}

\begin{proof}
Let \(N\ge 2\) and \(R(n)\in\Q(n)\) be as in Definition~\ref{def:summability-rational}, so that
\[
a_F(n)=R(n)
\qquad(n\ge N).
\]
Fix \(n\ge N\). Since \(R(n)\in\Q\), the number \(a_F(n)\) belongs to \(\Q\), hence is fixed by every element of \(G_n\).

Because \(F_{n-1}\) has coefficients in \(\Q\), for every \(\sigma\in G_n\) one has
\[
\sigma\!\Bigl(F_{n-1}(\alpha_{1,n},\dots,\alpha_{n-1,n})\Bigr)
=
F_{n-1}\bigl(\sigma(\alpha_{1,n}),\dots,\sigma(\alpha_{n-1,n})\bigr).
\]
Since the left-hand side equals \(\sigma(a_F(n))=a_F(n)\), the value \(a_F(n)\) is fixed by \(G_n\). Thus \(a_F(n)\in\Q\) for all sufficiently large \(n\).
\end{proof}

Lemma~\ref{lem:summability-implies-value-invariance} identifies the weakest arithmetic shadow of summability: eventual agreement with a rational function forces eventual rationality of the evaluated values. The next theorem shows that this condition alone has essentially no rigidity content.

\subsection{The non-converse: arbitrary rational tails}

The next result shows that compatible families can realize completely arbitrary rational-valued tails. In fact, the construction can be carried out with a uniform quadratic degree bound.

\begin{theorem}[Arbitrary rational tails can be interpolated in degree two]
\label{thm:no-converse}
Let \((b_n)_{n\ge 2}\) be any sequence with \(b_n\in \Q\) for all \(n\ge 2\).
Then there exists a compatible family \(F=(F_n)_{n\ge 1}\) such that
\[
\deg F_n\le 2\qquad(n\ge1)
\]
and
\[
F_{n-1}(\alpha_{1,n},\dots,\alpha_{n-1,n})=b_n
\qquad\text{for all }n\ge 2.
\]
In particular, value-level \(G_n\)-invariance does not imply summability at cosine points, even among compatible families of uniformly bounded total degree.
\end{theorem}

\begin{proof}
We construct the family recursively. The only exceptional level is \(4\), because
\(\alpha_{3,4}=0\); compatibility then forces the level-\(4\) value of \(F_3\) to be determined already by \(F_2\).

Set
\[
F_1(x_1)=b_2.
\]
Choose \(c\in\Q\) so that
\[
F_2^{(0)}(x_1,x_2)=b_2+c x_2
\]
satisfies
\[
F_2^{(0)}(\alpha_{1,3},\alpha_{2,3})=b_3.
\]
Since \(\alpha_{2,3}=-1/2\), we may take
\[
c=-2(b_3-b_2).
\]
Now
\[
(\alpha_{1,4},\alpha_{2,4})=(0,-1),
\qquad
\alpha_{2,3}=-\frac12.
\]
Define
\[
F_2(x_1,x_2)
=
F_2^{(0)}(x_1,x_2)
+d\,x_2(x_2-\alpha_{2,3}),
\]
where
\[
d=
\frac{b_4-F_2^{(0)}(0,-1)}{(-1)(-1-\alpha_{2,3})}
\in\Q.
\]
Then
\[
F_2(x_1,0)=F_1(x_1),
\]
while
\[
F_2(\alpha_{1,3},\alpha_{2,3})=b_3,
\qquad
F_2(\alpha_{1,4},\alpha_{2,4})=b_4.
\]
Set
\[
F_3(x_1,x_2,x_3)=F_2(x_1,x_2).
\]
Because \(\alpha_{3,4}=0\), this gives
\[
F_3(\alpha_{1,4},\alpha_{2,4},\alpha_{3,4})=b_4,
\]
and compatibility holds through level \(3\). Moreover, \(\deg F_j\le2\) for \(j=1,2,3\).

We now give the uniform step. Suppose \(n\ge4\) and \(F_{n-1}\in\Q[x_1,\dots,x_{n-1}]\) has been constructed with total degree at most \(2\). Put
\[
\Delta_{n+1}
=
b_{n+1}
-F_{n-1}(\alpha_{1,n+1},\dots,\alpha_{n-1,n+1}).
\]
Since all cosine values lie in the maximal real cyclotomic subfield,
\[
\Delta_{n+1}\in K_{n+1}^{+}:=\Q(\omega_{n+1})^{+}.
\]
Also
\[
\alpha_{n,n+1}
=\cos\!\left(\frac{2\pi n}{n+1}\right)
=\cos\!\left(\frac{2\pi}{n+1}\right)
=\alpha_{1,n+1}\ne0,
\]
because \(n+1\ge5\). Hence
\[
y:=\frac{\Delta_{n+1}}{\alpha_{n,n+1}}\in K_{n+1}^{+}.
\]

The real cyclotomic field is spanned over \(\Q\) by
\[
1,\alpha_{1,n+1},\dots,\alpha_{n-1,n+1}.
\]
Indeed, \(\Q(\omega_{n+1})\) is spanned by powers of \(\omega_{n+1}\), and averaging with complex conjugation shows that its fixed subfield is spanned by
\(1\) and the elements \(\omega_{n+1}^j+\omega_{n+1}^{-j}=2\alpha_{j,n+1}\).
Therefore there exist rational numbers \(c_0,c_1,\dots,c_{n-1}\) such that
\[
y=c_0+\sum_{j=1}^{n-1}c_j\alpha_{j,n+1}.
\]
Let
\[
L_n(x_1,\dots,x_{n-1})
=c_0+\sum_{j=1}^{n-1}c_jx_j
\]
and define
\[
F_n(x_1,\dots,x_n)
=
F_{n-1}(x_1,\dots,x_{n-1})
+x_nL_n(x_1,\dots,x_{n-1}).
\]
Then
\[
F_n(x_1,\dots,x_{n-1},0)=F_{n-1}(x_1,\dots,x_{n-1}),
\]
so compatibility is preserved, and
\begin{align*}
F_n(\alpha_{1,n+1},\dots,\alpha_{n,n+1})
&=F_{n-1}(\alpha_{1,n+1},\dots,\alpha_{n-1,n+1})
  +\alpha_{n,n+1}y\\
&=b_{n+1}.
\end{align*}
Finally, \(x_nL_n\) has total degree at most \(2\), so \(\deg F_n\le2\). The recursion therefore produces the required compatible family at every level.
\end{proof}

\begin{corollary}[Rational-valued non-summable compatible families exist]
\label{cor:prime-indicator-not-summable}
There exists a compatible family \(F=(F_n)_{n\ge 1}\) such that \(a_F(n)\in\Q\) for every \(n\ge 2\), but \(F\) is not summable at cosine points.
\end{corollary}

\begin{proof}
Apply Theorem~\ref{thm:no-converse} to the sequence
\[
b_n=
\begin{cases}
1,& \text{if \(n\) is prime},\\
0,& \text{if \(n\) is composite}.
\end{cases}
\]
Then \(a_F(n)=b_n\in\Q\) for all \(n\ge 2\). If \(F\) were summable at cosine points, then there would exist \(R(n)\in\Q(n)\) such that
\[
a_F(n)=R(n)
\]
for all sufficiently large integers \(n\). Hence
\[
R(n)(R(n)-1)=0
\]
for infinitely many primes and also for infinitely many composites. Therefore the rational function \(R(n)(R(n)-1)\) has infinitely many zeros, so it must vanish identically. It follows that \(R\equiv 0\) or \(R\equiv 1\), which is impossible. Thus \(F\) is not summable.
\end{proof}

\begin{remark}
\label{rem:sec2-transition}
Section~\ref{sec:summability-Gn} therefore identifies a sharp obstruction to any naive Galois-theoretic extension of the symmetric theory of \cite{VelezCadavidPartI}. Value-level \(G_n\)-invariance is a necessary arithmetic consequence of summability, but by itself it imposes no stable-range rigidity, no finite reduction to universal invariants, and no eventual rational-function behavior, even when one restricts to compatible families of total degree at most two.

To recover positive results, one must strengthen the hypotheses from a condition on the evaluated values to a structural condition on the formulas themselves. In the remainder of the paper we do this in a different framework: we leave the global compatible-family setting and pass to explicit orbit-structured pattern sums defined at admissible levels. No all-level compatibility is assumed there unless stated otherwise.
\end{remark}

\section{Orbit-structured pattern sums and finite generation}
\label{sec:Gp-patterns}

Having shown in Section~\ref{sec:summability-Gn} that value-level \(G_n\)-invariance is too weak to support any stable-range rigidity theorem, we now pass to a different positive framework. Instead of working with one truncation-compatible family across all integer levels, we introduce an explicit class of orbit-structured pattern sums defined at admissible levels and formulate the corresponding finite-generation hypothesis.

The role of the present section is preparatory. The all-level stable-range theorem proved in Section~\ref{sec:pattern-collapse} applies not to arbitrary non-symmetric formulas, but to families assembled from finitely many explicit pattern-sum building blocks through a single polynomial recipe. Thus the purpose here is simply to define the positive model class with precision.

\subsection{Admissible levels and multiplicative permutations}

Let \(n\ge 2\), and write
\[
\Omega_n=\{1,\dots,n-1\}.
\]
For each integer \(r\) with \(\gcd(r,n)=1\), multiplication by \(r\) modulo \(n\) induces a permutation of \(\Omega_n\), which we denote by
\[
\sigma_r:\Omega_n\longrightarrow\Omega_n,
\qquad
\sigma_r(k)\equiv rk \pmod n.
\]
Equivalently, \(\sigma_r\) is the permutation of \(\Omega_n\) arising from the Galois action of
\[
G_n=\Gal(\Q(\omega_n)/\Q)\cong (\Z/n\Z)^\ast
\]
on the punctured cyclotomic cosine points
\[
\alpha_{k,n}=\cos\!\Bigl(\frac{2\pi k}{n}\Bigr),
\qquad 1\le k\le n-1.
\]

Thus, even beyond full symmetry, one still has a natural arithmetic action on the punctured index set. The pattern sums introduced below are built from this action.

\subsection{Pattern sums}

We now define the basic building blocks of the positive theory.

\begin{definition}[Pattern sums at admissible levels]
\label{def:Gp-pattern-sums}
Fix integers
\[
t\ge 1,
\qquad
\mathbf r=(r_1,\dots,r_t)\in(\Z\setminus\{0\})^t,
\qquad
\mathbf h=(h_1,\dots,h_t)\in\Z_{\ge0}^t,
\]
with \(r_1=1\). Let \(n\ge 2\) be an integer such that
\[
\gcd(n,r_1\cdots r_t)=1.
\]
Then the associated \emph{pattern sum} is defined by
\[
T^{(n)}_{\mathbf r,\mathbf h}(x_1,\dots,x_{n-1})
=
\sum_{k\in\Omega_n}\ \prod_{j=1}^{t}\bigl(2x_{\sigma_{r_j}(k)}\bigr)^{h_j}
\in\Q[x_1,\dots,x_{n-1}].
\]
\end{definition}

\begin{remark}
\label{rem:pattern-degree}
The polynomial \(T^{(n)}_{\mathbf r,\mathbf h}\) has total degree
\[
|\mathbf h|:=h_1+\cdots+h_t.
\]
In general it is not invariant under the full symmetric group \(S_{n-1}\). Its symmetry is instead arithmetic: it is obtained by averaging one fixed monomial pattern along the multiplicative action of admissible residue classes modulo \(n\).
\end{remark}

\begin{remark}
\label{rem:prime-specialization}
When \(n=p\) is prime, \(\Omega_p\) is a single \(G_p\)-orbit, and Definition~\ref{def:Gp-pattern-sums} recovers the prime-level pattern sums that originally motivated the positive direction of the paper. The final form of the theory, however, is not confined to primes: the same construction makes sense at every admissible level.
\end{remark}

\subsection{Arithmetic invariance}

The pattern sums are not fully symmetric, but they are compatible with the arithmetic permutation action coming from the cyclotomic Galois group.

\begin{lemma}
\label{lem:Gp-pattern-invariant}
Let \(n\ge 2\) be an integer, and let \((\mathbf r,\mathbf h)\) be a pattern such that
\[
\gcd(n,r_1\cdots r_t)=1.
\]
Then the polynomial \(T^{(n)}_{\mathbf r,\mathbf h}\) is invariant under the natural action of \(G_n\) on \(\Omega_n\). Equivalently, for every \(a\) with \(\gcd(a,n)=1\),
\[
T^{(n)}_{\mathbf r,\mathbf h}(x_1,\dots,x_{n-1})
=
T^{(n)}_{\mathbf r,\mathbf h}(x_{\sigma_a(1)},\dots,x_{\sigma_a(n-1)}).
\]
\end{lemma}

\begin{proof}
Fix \(a\) with \(\gcd(a,n)=1\). Since \(\sigma_a\) permutes \(\Omega_n\), changing variables \(k\mapsto \sigma_a(k)\) preserves the index set. Moreover, for each \(j\),
\[
\sigma_{r_j}(\sigma_a(k))=\sigma_a(\sigma_{r_j}(k)),
\]
because both sides are congruent to \(ar_jk\pmod n\). Thus the monomials appearing in the defining sum are merely permuted by the action of \(\sigma_a\), and the total sum is unchanged.
\end{proof}

\begin{remark}
\label{rem:orbit-structured-not-symmetric}
Lemma~\ref{lem:Gp-pattern-invariant} shows that the pattern sums belong to a genuinely arithmetic symmetry class. They need not be invariant under arbitrary permutations of the variables, but they are invariant under the much smaller group generated by multiplication modulo the level. This is precisely the symmetry retained from the cyclotomic Galois action once full symmetric-group invariance is abandoned.
\end{remark}

\subsection{Uniform finite generation}

A single pattern sum already provides a structured non-symmetric invariant, but the positive theorem of Section~\ref{sec:pattern-collapse} applies to families generated from finitely many such patterns through a single polynomial rule independent of the level. This is the replacement, in the present framework, for the finite-invariant philosophy of the symmetric theory.

\begin{definition}[Uniform finite pattern generation]
\label{def:UPGp}
A family
\[
\{F^{(n)}\}_{n\ge 2},
\qquad
F^{(n)}\in\Q[x_1,\dots,x_{n-1}],
\]
is said to be \emph{uniformly pattern-generated} if there exist:
\begin{itemize}
\item an integer \(n_0\),
\item finitely many fixed patterns
\[
(\mathbf r^{(1)},\mathbf h^{(1)}),\dots,(\mathbf r^{(M)},\mathbf h^{(M)}),
\]
and
\item a single polynomial
\[
\Psi\in\Q[z_1,\dots,z_M],
\]
\end{itemize}
such that for every integer \(n\ge n_0\) satisfying
\[
\gcd\!\Bigl(n,\prod_{m=1}^M\prod_j r_j^{(m)}\Bigr)=1,
\]
one has
\[
F^{(n)}(x_1,\dots,x_{n-1})
=
\Psi\Bigl(
T^{(n)}_{\mathbf r^{(1)},\mathbf h^{(1)}}(x),\dots,
T^{(n)}_{\mathbf r^{(M)},\mathbf h^{(M)}}(x)
\Bigr).
\]
\end{definition}

\begin{remark}
\label{rem:uniform-pattern-philosophy}
Definition~\ref{def:UPGp} is intentionally explicit. It does not claim to characterize all orbit-structured formulas of interest. Rather, it isolates a concrete finitely generated model class on which one can prove a complete stable-range theorem. Notice that no bound on the total degree of \(F^{(n)}\) is being asserted: the degrees of the finitely many pattern generators are fixed, but the polynomial recipe \(\Psi\) may increase the total degree. The point is not maximal generality, but a precise positive framework strong enough to recover eventual polynomiality beyond the symmetric setting.
\end{remark}

\subsection{The evaluation problem}

For a uniformly pattern-generated family \(\{F^{(n)}\}\), define its cyclotomic cosine evaluation by
\[
a(n)=F^{(n)}(\alpha_{1,n},\dots,\alpha_{n-1,n}),
\qquad
\alpha_{k,n}=\cos\!\Bigl(\frac{2\pi k}{n}\Bigr).
\]
The positive theory now reduces to the evaluation of one fixed pattern sum at a time. If one can show that each
\[
T^{(n)}_{\mathbf r,\mathbf h}(\alpha_{1,n},\dots,\alpha_{n-1,n})
\]
collapses, for all sufficiently large admissible levels \(n\), to an affine function of \(n\), then every uniformly pattern-generated family inherits eventual polynomiality by finite substitution into the single polynomial recipe \(\Psi\).

Section~\ref{sec:pattern-collapse} proves exactly this: a stable-range all-level collapse theorem for every fixed pattern, together with the corresponding polynomiality theorem for finitely generated families. The prime, semiprime, and prime-power cases then appear as orbitwise refinements of that general result.

\section{Stable-range collapse and orbitwise refinements}
\label{sec:pattern-collapse}

This section contains the main positive results of the paper. The key point is that, for the total pattern sum, one can bypass the orbit decomposition entirely. After binomial expansion, the evaluation reduces to finitely many geometric sums indexed by integer frequencies. In the stable range, every nonzero frequency is too small in absolute value to be divisible by the level \(n\), so all nonzero modes contribute the same constant \(-1\), while the zero mode contributes \(n-1\). This yields an all-level collapse theorem valid for every sufficiently large admissible level.

We next derive a general formula for each gcd-orbit in terms of Ramanujan sums. The semiprime and prime-power formulas then follow as particularly simple refinements. These orbitwise results are no longer needed to prove the total collapse formula, but they exhibit the internal arithmetic distribution of the total contribution.

\subsection{All-level stable-range collapse}

We retain the notation of Section~\ref{sec:Gp-patterns}. Thus \((\mathbf r,\mathbf h)\) denotes a fixed pattern with
\[
\mathbf r=(r_1,\dots,r_t)\in(\Z\setminus\{0\})^t,
\qquad
\mathbf h=(h_1,\dots,h_t)\in\Z_{\ge0}^t,
\qquad r_1=1.
\]
Write
\[
H=|\mathbf h|=h_1+\cdots+h_t,
\qquad
B(\mathbf r,\mathbf h)=\sum_{j=1}^t h_j|r_j|.
\]
For brevity, write
\[
\boldsymbol\alpha_n=(\alpha_{1,n},\dots,\alpha_{n-1,n}).
\]
The finite Laurent polynomial
\[
A_{\mathbf r,\mathbf h}(u)
:=
\prod_{j=1}^t (u^{r_j}+u^{-r_j})^{h_j}
\]
will be written as
\begin{equation}\label{eq:frequency-expansion}
A_{\mathbf r,\mathbf h}(u)
=
\sum_{M=-B}^{B}\mathsf S_M(\mathbf r,\mathbf h)u^M.
\end{equation}
Explicitly,
\[
\mathsf S_M(\mathbf r,\mathbf h)
=
\sum_{\substack{m_1,\dots,m_t\in\Z\\
|m_j|\le h_j,\ m_j\equiv h_j\!\!\pmod 2\\
\sum_{j=1}^t m_jr_j=M}}
\prod_{j=1}^t
\binom{h_j}{(h_j+m_j)/2}.
\]
In particular, \(\mathsf S_0(\mathbf r,\mathbf h)\) is the zero-frequency coefficient. Evaluating \eqref{eq:frequency-expansion} at \(u=1\) gives the useful identity
\begin{equation}\label{eq:sum-frequency-coeffs}
\sum_{M=-B}^{B}\mathsf S_M(\mathbf r,\mathbf h)
=A_{\mathbf r,\mathbf h}(1)
=2^H.
\end{equation}

\begin{theorem}[All-level collapse in the stable range]
\label{thm:all-level-collapse}
Fix a pattern \((\mathbf r,\mathbf h)\), and let
\[
H=|\mathbf h|,
\qquad
B=B(\mathbf r,\mathbf h).
\]
Let \(n\ge2\) satisfy
\[
n>B
\qquad\text{and}\qquad
\gcd(n,r_1\cdots r_t)=1.
\]
Then
\[
T^{(n)}_{\mathbf r,\mathbf h}(\alpha_{1,n},\dots,\alpha_{n-1,n})
=
n\,\mathsf S_0(\mathbf r,\mathbf h)-2^H.
\]
In particular, for every fixed pattern, the cyclotomic cosine evaluation is an affine function of \(n\) throughout the stable range.
\end{theorem}

\begin{proof}
Let \(\zeta_n=e^{2\pi i/n}\). Since the level is admissible,
\[
2\alpha_{\sigma_{r_j}(k),n}
=
\zeta_n^{r_jk}+\zeta_n^{-r_jk}.
\]
Therefore, using \eqref{eq:frequency-expansion},
\begin{align*}
T^{(n)}_{\mathbf r,\mathbf h}(\alpha_{1,n},\dots,\alpha_{n-1,n})
&=
\sum_{k=1}^{n-1}
A_{\mathbf r,\mathbf h}(\zeta_n^k)\\
&=
\sum_{M=-B}^{B}\mathsf S_M(\mathbf r,\mathbf h)
\sum_{k=1}^{n-1}\zeta_n^{Mk}.
\end{align*}
Because \(n>B\), every nonzero frequency in the Laurent polynomial satisfies
\(0<|M|<n\). Hence
\[
\sum_{k=1}^{n-1}\zeta_n^{Mk}
=
\begin{cases}
n-1,&M=0,\\
-1,&M\ne0.
\end{cases}
\]
It follows that
\begin{align*}
T^{(n)}_{\mathbf r,\mathbf h}(\boldsymbol\alpha_n)
&=(n-1)\mathsf S_0
   -\sum_{M\ne0}\mathsf S_M\\
&=n\mathsf S_0-\sum_M\mathsf S_M\\
&=n\mathsf S_0-2^H,
\end{align*}
where the last equality is \eqref{eq:sum-frequency-coeffs}.
\end{proof}

\begin{remark}
\label{rem:all-level-collapse-meaning}
The proof isolates the mechanism completely. The finite Laurent polynomial \(A_{\mathbf r,\mathbf h}\) contains only frequencies \(|M|\le B\). Once \(n>B\), no nonzero frequency can be divisible by \(n\); the zero mode contributes \(n-1\), while every other mode contributes \(-1\). Thus the stable-range evaluation is controlled by the single zero-frequency coefficient together with the universal coefficient sum \(2^H\).
\end{remark}

\subsection{Polynomiality for uniformly pattern-generated families}

The finite-generation hypothesis from Section~\ref{sec:Gp-patterns} now yields a stronger positive theorem than the prime-tail statement proved earlier: one obtains polynomiality for every sufficiently large admissible level, not merely on the prime tail.

\begin{theorem}[Eventual polynomiality at all admissible levels]
\label{thm:all-level-polynomiality}
Let \(\{F^{(n)}\}_{n\ge2}\) be uniformly pattern-generated in the sense of Definition~\ref{def:UPGp}. Thus, for some integer \(n_0\), fixed patterns
\[
(\mathbf r^{(1)},\mathbf h^{(1)}),\dots,(\mathbf r^{(M)},\mathbf h^{(M)}),
\]
and a polynomial \(\Psi\in\Q[z_1,\dots,z_M]\), one has
\[
F^{(n)}(x)
=
\Psi\Bigl(
T^{(n)}_{\mathbf r^{(1)},\mathbf h^{(1)}}(x),\dots,
T^{(n)}_{\mathbf r^{(M)},\mathbf h^{(M)}}(x)
\Bigr)
\]
for every sufficiently large admissible \(n\). Define
\[
a(n)=F^{(n)}(\alpha_{1,n},\dots,\alpha_{n-1,n}).
\]
Then there exist an integer \(n_1\) and a polynomial
\[
R_\infty(n)\in\Q[n]
\]
such that
\[
a(n)=R_\infty(n)
\]
for every \(n\ge n_1\) satisfying
\[
\gcd\!\Bigl(n,\prod_{m=1}^M\prod_j r_j^{(m)}\Bigr)=1.
\]
\end{theorem}

\begin{proof}
For \(m=1,\dots,M\), put
\[
H_m=|\mathbf h^{(m)}|,
\qquad
B_m=B(\mathbf r^{(m)},\mathbf h^{(m)}).
\]
Choose \(n_1\ge n_0\) with \(n_1>\max_m B_m\). If \(n\ge n_1\) is admissible for all the fixed patterns, Theorem~\ref{thm:all-level-collapse} gives
\[
T^{(n)}_{\mathbf r^{(m)},\mathbf h^{(m)}}(\boldsymbol\alpha_n)
=A_m n-C_m,
\]
where
\[
A_m=\mathsf S_0(\mathbf r^{(m)},\mathbf h^{(m)})\in\Z_{\ge0},
\qquad
C_m=2^{H_m}.
\]
Hence
\[
a(n)=\Psi(A_1n-C_1,\dots,A_Mn-C_M).
\]
The right-hand side is a polynomial in \(n\). Thus one may take
\[
R_\infty(n)=\Psi(A_1n-C_1,\dots,A_Mn-C_M)\in\Q[n].
\]
\end{proof}

\begin{remark}
\label{rem:all-level-vs-prime-tail}
Theorem~\ref{thm:all-level-polynomiality} strictly strengthens the prime-tail polynomiality statement. The prime-tail result is recovered by restricting to prime values of \(n\), but the same collapse mechanism applies in fact to every sufficiently large admissible level.
\end{remark}

\subsection{Prime and prime-tail consequences}

The prime-tail theorem is an immediate corollary of the all-level collapse theorem.

\begin{corollary}[Prime-tail collapse]
\label{cor:prime-tail-collapse}
Fix a pattern \((\mathbf r,\mathbf h)\), and let
\[
H=|\mathbf h|,
\qquad
B=B(\mathbf r,\mathbf h).
\]
Then for every prime \(p>B\) such that
\[
p\nmid r_1\cdots r_t,
\]
one has
\[
T^{(p)}_{\mathbf r,\mathbf h}(\alpha_{1,p},\dots,\alpha_{p-1,p})
=
p\,\mathsf S_0(\mathbf r,\mathbf h)-2^H.
\]
\end{corollary}

\begin{proof}
Apply Theorem~\ref{thm:all-level-collapse} with \(n=p\).
\end{proof}

\begin{corollary}[Prime-tail polynomiality]
\label{cor:prime-tail-polynomiality}
Let
\[
\{F^{(p)}\}_{p\ \mathrm{prime}}
\]
be the restriction to prime levels of a uniformly pattern-generated family in the sense of Definition~\ref{def:UPGp}. Define
\[
a(p)=F^{(p)}(\alpha_{1,p},\dots,\alpha_{p-1,p}).
\]
Then there exist a prime threshold \(p_1\) and a polynomial
\[
R_\infty(p)\in\Q[p]
\]
such that
\[
a(p)=R_\infty(p)
\qquad\text{for all primes }p\ge p_1.
\]
\end{corollary}

\begin{proof}
This follows immediately from Theorem~\ref{thm:all-level-polynomiality} by restricting to prime values of \(n\).
\end{proof}

\subsection{An explicit non-symmetric example}

We record a simple example showing that the theory applies to families which are genuinely arithmetic and not fully symmetric.

\begin{example}
\label{ex:prime-tail-example}
For every integer \(n\ge2\) with \(\gcd(n,6)=1\), define
\[
Z_1^{(n)}(x_1,\dots,x_{n-1})
=
\sum_{k=1}^{n-1}(2x_k)(2x_{\sigma_2(k)})(2x_{\sigma_3(k)})
\]
and
\[
Z_2^{(n)}(x_1,\dots,x_{n-1})
=
\sum_{k=1}^{n-1}(2x_k)^2(2x_{\sigma_2(k)}).
\]
Set
\[
F^{(n)}(x)=Z_1^{(n)}(x)Z_2^{(n)}(x).
\]
Then \(\{F^{(n)}\}\) is uniformly pattern-generated on the admissible levels, and for every \(n>6\) with \(\gcd(n,6)=1\),
\[
F^{(n)}(\alpha_{1,n},\dots,\alpha_{n-1,n})=4(n-4)^2.
\]
\end{example}

\begin{proof}
The first generator corresponds to
\[
\mathbf r=(1,2,3),
\qquad
\mathbf h=(1,1,1).
\]
Here \(H=3\), \(B=6\), and the zero-frequency equation
\[
m_1+2m_2+3m_3=0,
\qquad m_j\in\{\pm1\},
\]
has exactly the two solutions
\[
(1,1,-1),\qquad(-1,-1,1).
\]
Hence \(\mathsf S_0=2\), and Theorem~\ref{thm:all-level-collapse} gives
\[
Z_1^{(n)}(\boldsymbol\alpha_n)=2n-8
\]
for every admissible \(n>6\).

For \(Z_2^{(n)}\), use the pattern
\[
\mathbf r=(1,2),
\qquad
\mathbf h=(2,1).
\]
Then \(H=3\), \(B=4\), and the zero-frequency condition is
\[
m_1+2m_2=0,
\qquad
m_1\in\{-2,0,2\},\quad m_2\in\{\pm1\}.
\]
The only solutions are \((-2,1)\) and \((2,-1)\), each with coefficient \(1\). Thus \(\mathsf S_0=2\), so
\[
Z_2^{(n)}(\boldsymbol\alpha_n)=2n-8
\]
for every admissible \(n>4\). Consequently, for \(n>6\) and \(\gcd(n,6)=1\),
\[
F^{(n)}(\boldsymbol\alpha_n)
=(2n-8)^2
=4(n-4)^2.
\]
\end{proof}

\begin{remark}
\label{rem:example-not-symmetric}
Example~\ref{ex:prime-tail-example} lies genuinely outside the symmetric bounded-degree framework of \cite{VelezCadavidPartI}. Its eventual polynomial behavior is forced by orbit structure rather than by full permutation symmetry.
\end{remark}

\subsection{General gcd-orbits and Ramanujan sums}
\label{subsec:general-orbits}

The orbitwise structure admits an exact formula at every admissible level, not only for semiprimes and prime powers. For \(d\mid n\), \(d<n\), set
\[
\mathcal O_d(n)=\{k\in\Omega_n:\gcd(k,n)=d\}
\]
and define
\[
T^{(n;d)}_{\mathbf r,\mathbf h}(x)
=
\sum_{k\in\mathcal O_d(n)}
\prod_{j=1}^t(2x_{\sigma_{r_j}(k)})^{h_j}.
\]
We use the classical Ramanujan sum \cite{Ramanujan1918,Apostol1976,McCarthy1986}
\[
c_q(M)=\sum_{u\in(\Z/q\Z)^\ast}e^{2\pi iMu/q}.
\]

\begin{proposition}[Exact orbitwise Ramanujan formula]
\label{prop:general-orbit-ramanujan}
Let \((\mathbf r,\mathbf h)\) be a fixed pattern, set \(B=B(\mathbf r,\mathbf h)\), and let \(n\) be an admissible level. If \(d\mid n\), \(d<n\), and \(q=n/d\), then
\[
T^{(n;d)}_{\mathbf r,\mathbf h}(\alpha_{1,n},\dots,\alpha_{n-1,n})
=
\sum_{M=-B}^{B}
\mathsf S_M(\mathbf r,\mathbf h)c_q(M).
\]
\end{proposition}

\begin{proof}
Every \(k\in\mathcal O_d(n)\) has a unique representation
\[
k=du,
\qquad
u\in(\Z/q\Z)^\ast,
\]
with \(q=n/d\). Hence
\[
\alpha_{\sigma_{r_j}(k),n}
=
\cos\!\left(\frac{2\pi r_jdu}{n}\right)
=
\cos\!\left(\frac{2\pi r_ju}{q}\right).
\]
Writing \(\zeta_q=e^{2\pi i/q}\) and using \eqref{eq:frequency-expansion}, we obtain
\begin{align*}
T^{(n;d)}_{\mathbf r,\mathbf h}(\boldsymbol\alpha_n)
&=
\sum_{u\in(\Z/q\Z)^\ast}
A_{\mathbf r,\mathbf h}(\zeta_q^u)\\
&=
\sum_{M=-B}^{B}\mathsf S_M(\mathbf r,\mathbf h)
\sum_{u\in(\Z/q\Z)^\ast}\zeta_q^{Mu}\\
&=
\sum_{M=-B}^{B}\mathsf S_M(\mathbf r,\mathbf h)c_q(M).
\end{align*}
\end{proof}

\begin{remark}
The standard identity
\[
c_q(M)
=
\mu\!\left(\frac{q}{(q,M)}\right)
\frac{\varphi(q)}{\varphi(q/(q,M))}
\]
(with the usual interpretation when \(M=0\)) makes the dependence on the arithmetic type of the orbit explicit. In particular, the semiprime and prime-power formulas below are immediate simplifications of Proposition~\ref{prop:general-orbit-ramanujan}.
\end{remark}

\subsection{Semiprime orbitwise refinement}

Although Theorem~\ref{thm:all-level-collapse} already yields the total formula for admissible semiprime levels, the orbitwise decomposition remains informative because it shows how the total contribution is distributed across the distinct \(G_{pq}\)-orbits.

Let \(p\neq q\) be primes, and write
\[
n=pq.
\]
Then
\[
\Omega_{pq}=\{1,\dots,pq-1\}
=
\mathcal O_1(pq)\sqcup \mathcal O_p(pq)\sqcup \mathcal O_q(pq),
\]
where
\[
\mathcal O_d(pq)=\{\,k\in\Omega_{pq}:\gcd(k,pq)=d\,\},
\qquad d\in\{1,p,q\}.
\]
For a fixed pattern \((\mathbf r,\mathbf h)\), define the orbitwise sums
\[
T^{(pq;d)}_{\mathbf r,\mathbf h}(x_1,\dots,x_{pq-1})
=
\sum_{k\in \mathcal O_d(pq)}
\prod_{j=1}^t (2x_{\sigma_{r_j}(k)})^{h_j},
\qquad d\in\{1,p,q\}.
\]

\begin{theorem}[Semiprime orbitwise refinement]
\label{thm:semiprime-orbit-collapse}
Fix a pattern \((\mathbf r,\mathbf h)\), and let
\[
H=|\mathbf h|,
\qquad
B=B(\mathbf r,\mathbf h).
\]
Assume that \(p\neq q\) are primes such that
\[
p>B,\qquad q>B,\qquad p\nmid r_1\cdots r_t,\qquad q\nmid r_1\cdots r_t.
\]
Then
\begin{align*}
T^{(pq;p)}_{\mathbf r,\mathbf h}(\alpha_{1,pq},\dots,\alpha_{pq-1,pq})
&=
q\,\mathsf S_0(\mathbf r,\mathbf h)-2^H,\\[2mm]
T^{(pq;q)}_{\mathbf r,\mathbf h}(\alpha_{1,pq},\dots,\alpha_{pq-1,pq})
&=
p\,\mathsf S_0(\mathbf r,\mathbf h)-2^H,\\[2mm]
T^{(pq;1)}_{\mathbf r,\mathbf h}(\alpha_{1,pq},\dots,\alpha_{pq-1,pq})
&=
\bigl((p-1)(q-1)-1\bigr)\mathsf S_0(\mathbf r,\mathbf h)+2^H.
\end{align*}
Consequently,
\[
T^{(pq)}_{\mathbf r,\mathbf h}(\alpha_{1,pq},\dots,\alpha_{pq-1,pq})
=
pq\,\mathsf S_0(\mathbf r,\mathbf h)-2^H.
\]
\end{theorem}

\begin{proof}
For the orbit \(\mathcal O_p(pq)\), one has \(q=(pq)/p\). Since \(q>B\), Proposition~\ref{prop:general-orbit-ramanujan} and
\[
c_q(0)=q-1,
\qquad
c_q(M)=-1\quad(0<|M|<q)
\]
give
\[
T^{(pq;p)}_{\mathbf r,\mathbf h}(\boldsymbol\alpha_{pq})
=q\mathsf S_0-2^H.
\]
The same argument with \(p\) and \(q\) interchanged gives
\[
T^{(pq;q)}_{\mathbf r,\mathbf h}(\boldsymbol\alpha_{pq})
=p\mathsf S_0-2^H.
\]

For the unit orbit, the relevant Ramanujan sum has modulus \(pq\). If \(M\ne0\), then
\(|M|\le B<\min\{p,q\}\), so \((M,pq)=1\), and therefore
\[
c_{pq}(M)=\mu(pq)=1.
\]
Also \(c_{pq}(0)=\varphi(pq)=(p-1)(q-1)\). Hence
\begin{align*}
T^{(pq;1)}_{\mathbf r,\mathbf h}(\boldsymbol\alpha_{pq})
&=(p-1)(q-1)\mathsf S_0+
  \sum_{M\ne0}\mathsf S_M\\
&=\bigl((p-1)(q-1)-1\bigr)\mathsf S_0+2^H,
\end{align*}
using \eqref{eq:sum-frequency-coeffs}. Adding the three orbit contributions yields
\[
T^{(pq)}_{\mathbf r,\mathbf h}(\boldsymbol\alpha_{pq})
=pq\,\mathsf S_0-2^H,
\]
as required.
\end{proof}

\begin{remark}
\label{rem:semiprime-interpretation}
Theorem~\ref{thm:semiprime-orbit-collapse} no longer serves as the main positive theorem, since the total semiprime collapse already follows from Theorem~\ref{thm:all-level-collapse}. Its role is now finer: it exhibits the internal distribution of the total contribution across the three semiprime orbit types. In particular, it shows that the boundary orbits reduce to prime-level sums, whereas the unit orbit contributes through a Ramanujan-sum mechanism.
\end{remark}

\subsection{Prime-power orbitwise refinement}

Prime powers provide a second natural refinement. Again, the total formula follows immediately from the all-level theorem, but the orbitwise decomposition is structurally revealing.

Let \(N\ge 1\), and write
\[
n=p^N.
\]
Then
\[
\Omega_{p^N}
=
\bigsqcup_{s=0}^{N-1}\mathcal O_{p^s}(p^N),
\qquad
\mathcal O_{p^s}(p^N)=\{\,k\in\Omega_{p^N}:\gcd(k,p^N)=p^s\,\}.
\]
For a fixed pattern \((\mathbf r,\mathbf h)\), define
\[
T^{(p^N;p^s)}_{\mathbf r,\mathbf h}(x_1,\dots,x_{p^N-1})
=
\sum_{k\in\mathcal O_{p^s}(p^N)}
\prod_{j=1}^t (2x_{\sigma_{r_j}(k)})^{h_j},
\qquad 0\le s\le N-1.
\]

\begin{theorem}[Prime-power orbitwise refinement]
\label{thm:prime-power-orbit-collapse}
Fix a pattern \((\mathbf r,\mathbf h)\), and let
\[
H=|\mathbf h|,
\qquad
B=B(\mathbf r,\mathbf h).
\]
Let \(N\ge 1\). Assume that \(p\) is a prime such that
\[
p>B
\qquad\text{and}\qquad
p\nmid r_1\cdots r_t.
\]
Then
\begin{align*}
T^{(p^N;p^{N-1})}_{\mathbf r,\mathbf h}(\alpha_{1,p^N},\dots,\alpha_{p^N-1,p^N})
&=
p\,\mathsf S_0(\mathbf r,\mathbf h)-2^H,\\[2mm]
T^{(p^N;p^s)}_{\mathbf r,\mathbf h}(\alpha_{1,p^N},\dots,\alpha_{p^N-1,p^N})
&=
\varphi(p^{N-s})\,\mathsf S_0(\mathbf r,\mathbf h)
\qquad (0\le s\le N-2).
\end{align*}
Consequently,
\[
T^{(p^N)}_{\mathbf r,\mathbf h}(\alpha_{1,p^N},\dots,\alpha_{p^N-1,p^N})
=
p^N\,\mathsf S_0(\mathbf r,\mathbf h)-2^H.
\]
\end{theorem}

\begin{proof}
For \(s=N-1\), the quotient modulus in Proposition~\ref{prop:general-orbit-ramanujan} is \(p\). Since \(p>B\),
\[
c_p(0)=p-1,
\qquad
c_p(M)=-1\quad(M\ne0),
\]
and hence
\[
T^{(p^N;p^{N-1})}_{\mathbf r,\mathbf h}(\boldsymbol\alpha_{p^N})
=p\mathsf S_0-2^H.
\]

Now let \(0\le s\le N-2\). The quotient modulus is \(p^{N-s}\) with exponent at least \(2\). If \(M\ne0\), then \(|M|\le B<p\), so \(p\nmid M\); consequently
\[
c_{p^{N-s}}(M)=0.
\]
For \(M=0\),
\[
c_{p^{N-s}}(0)=\varphi(p^{N-s}).
\]
Proposition~\ref{prop:general-orbit-ramanujan} therefore gives
\[
T^{(p^N;p^s)}_{\mathbf r,\mathbf h}(\boldsymbol\alpha_{p^N})
=\varphi(p^{N-s})\mathsf S_0
\qquad(0\le s\le N-2).
\]
Finally,
\[
p+\sum_{s=0}^{N-2}\varphi(p^{N-s})
=p+\sum_{m=2}^{N}(p^m-p^{m-1})
=p^N,
\]
so summing the orbit contributions yields
\[
T^{(p^N)}_{\mathbf r,\mathbf h}(\boldsymbol\alpha_{p^N})
=p^N\mathsf S_0-2^H.
\]
\end{proof}

\begin{remark}
\label{rem:prime-power-interpretation}
The prime-power refinement shows that the constant term \(-2^H\) arises entirely from the deepest boundary orbit \(\mathcal O_{p^{N-1}}(p^N)\), which reduces to the prime-level case. All inner orbits contribute only zero-frequency mass. Thus, although the total collapse formula is uniform in the stable range, the internal orbitwise distribution of the contribution depends sensitively on the arithmetic shape of the level.
\end{remark}

\section{Conclusion and outlook}
\label{sec:conclusion}

The results proved in Sections~\ref{sec:summability-Gn}--\ref{sec:pattern-collapse} clarify, in a precise sense, how far one can move beyond the fully symmetric framework of \cite{VelezCadavidPartI} while retaining summability phenomena at cyclotomic cosine points.

The first conclusion is negative but decisive. Within the compatible-family framework, summability by a rational function of the level forces eventual rationality of the evaluated values, hence eventual value-level \(G_n\)-invariance. However, this arithmetic shadow is far too weak to imply any rigidity. Indeed, every rational-valued sequence can occur as the cosine-point evaluation of a compatible family of total degree at most two. Thus no stable-range theorem beyond symmetry can be recovered merely from value-level Galois invariance, even under a uniform quadratic degree bound.

The second conclusion is positive and substantially stronger than the prime-tail and semiprime statements that initially motivated the investigation. For the explicit orbit-structured pattern class introduced in Section~\ref{sec:Gp-patterns}, the total evaluated sum satisfies a stable-range all-level collapse formula at every sufficiently large admissible level \(n\):
\[
T^{(n)}_{\mathbf r,\mathbf h}(\alpha_{1,n},\dots,\alpha_{n-1,n})
=
n\,\mathsf S_0(\mathbf r,\mathbf h)-2^{|\mathbf h|}.
\]
Thus the positive theory is not confined to primes or to a few special factorizations of the level. Rather, for this concrete non-symmetric class, eventual affine behavior persists uniformly across all sufficiently large levels for which the defining multipliers remain invertible modulo \(n\).

As a consequence, any family uniformly generated from finitely many such pattern sums through a single polynomial recipe has eventually polynomial evaluation at all sufficiently large admissible levels. In this sense, the paper establishes a genuine non-symmetric stable-range phenomenon: although full permutation symmetry is absent, the total contribution is still governed by a finite zero-frequency invariant together with a universal constant term.

The paper also gives an exact formula for every gcd-orbit in terms of Ramanujan sums. The prime, semiprime, and prime-power cases remain mathematically meaningful after the all-level theorem is proved, but their role changes: they become especially transparent specializations of the general orbitwise formula rather than separate main theorems.

At semiprime levels \(n=pq\), the total collapse decomposes into three orbit contributions. The two non-unit orbits reduce to prime-level sums, while the unit orbit is governed by a Ramanujan-sum contribution over \((\Z/pq\Z)^\ast\). At prime-power levels \(n=p^N\), the punctured index set breaks into a tower of gcd-orbits, with the deepest boundary orbit contributing the unique constant term \(-2^{|\mathbf h|}\) and the inner orbits contributing only zero-frequency mass. These refinements show that, although the total collapse theorem is uniform, the internal orbitwise arithmetic still depends strongly on the factorization type of the level.

\subsection{Beyond the total collapse formula}

For the pattern class studied here, both the total sum and every individual gcd-orbit are now explicit: the former collapses in the stable range, while the latter is governed by Proposition~\ref{prop:general-orbit-ramanujan}. The remaining problem is therefore genuinely broader. One would like to understand for which larger classes of orbit-structured families comparable uniform control of the separate orbit contributions is possible.

For a general composite integer \(n\), the punctured index set decomposes as
\[
\Omega_n=\{1,\dots,n-1\}
=
\bigsqcup_{\substack{d\mid n\\ d<n}}\mathcal O_d(n),
\qquad
\mathcal O_d(n)=\{\,k\in\Omega_n:\gcd(k,n)=d\,\}.
\]
At bounded degree, any orbit-structured theory should naturally involve orbitwise quantities such as
\[
P_{h,d}(n)
=
\sum_{k\in\mathcal O_d(n)}(2\alpha_{k,n})^h,
\qquad
\alpha_{k,n}=\cos\!\Bigl(\frac{2\pi k}{n}\Bigr).
\]
The symmetric theory of \cite{VelezCadavidPartI} bypasses these objects because full symmetry erases orbit distinctions. The present paper bypasses them at the level of the total pattern sum because the stable-range frequency argument collapses all nonzero modes uniformly; for the explicit pattern class, Proposition~\ref{prop:general-orbit-ramanujan} also resolves the individual gcd-orbits in terms of Ramanujan sums. What remains open is an orbit-elimination theory for substantially broader structural classes. The semiprime and prime-power refinements illustrate how differently individual orbit types may contribute even when the total answer is simple.

\subsection{Twists and rational extensions}

A second natural direction is to incorporate twists. In place of the untwisted pattern sums considered here, one may study character-weighted expressions such as
\[
\sum_{k=1}^{n-1}\chi(k)\prod_{j=1}^{t}(2\alpha_{\sigma_{r_j}(k),n})^{h_j},
\]
where \(\chi\) is a Dirichlet character or another natural multiplicative weight. From the viewpoint of the present paper, such twists break the uniform collapse of the nonzero modes and lead naturally to Fourier-weighted or isotypic orbit sums.

A third direction is to enlarge the class of admissible formulas from polynomials to rational functions. The symmetric theory of \cite{VelezCadavidPartI} yields eventual polynomiality for bounded-degree symmetric polynomial families, while the present paper shows that an explicit non-symmetric orbit-structured class still satisfies an all-level stable-range collapse theorem. It is natural to ask for an eventual rational analogue of these results when the input family is rational rather than polynomial, under suitable nonvanishing hypotheses on the denominator along the cyclotomic cosine grid.

We do not pursue these directions here. The role of the present paper is more modest and foundational: it identifies a necessary arithmetic shadow of summability, proves that this shadow does not characterize summability, and establishes a broad stable-range collapse mechanism for an explicit class of non-symmetric orbit-structured formulas, together with an exact Ramanujan-sum description of its gcd-orbits. In this sense, the paper should be viewed as a first step from the fully symmetric inverse-limit world of \cite{VelezCadavidPartI} toward a broader orbit-theoretic theory in which the model class is understood both globally and orbitwise, while extensions to wider structural classes remain to be developed.

\section*{Acknowledgments}
Carlos A. Cadavid gratefully acknowledges the financial support of Universidad EAFIT (Colombia) for the project \emph{Study and Applications of Diffusion Processes of Importance in Health and Computation} (project code 11740052022). Juan D. V\'elez gratefully acknowledges the Universidad Nacional de Colombia for its support during this research.

\section*{Statements and Declarations}
\noindent\textbf{Funding.} Carlos A. Cadavid received financial support from Universidad EAFIT (Colombia) for the project \emph{Study and Applications of Diffusion Processes of Importance in Health and Computation} (project code 11740052022). Juan D. V\'elez received support from Universidad Nacional de Colombia during this research.

\medskip
\noindent\textbf{Competing interests.} The authors declare that they have no competing interests.

\medskip
\noindent\textbf{Data availability.} No datasets were generated or analyzed during the current study.

\end{document}